\documentclass[12pt, reqno]{amsart}
\usepackage{amsfonts,amssymb,amscd,amstext,mathrsfs,color}

\usepackage{pgf,tikz}
\usetikzlibrary{positioning,shapes,decorations.markings}
\usetikzlibrary{arrows}
\usetikzlibrary[patterns]
\usetikzlibrary{cd}

\input xy
\xyoption{all}

\newtheorem{theorem}{Theorem}

\newtheorem{lemma}{Lemma}
\newtheorem{proposition}{Proposition}

\theoremstyle{definition}

\newcommand{\C}{\mathbb{C}}
\newcommand{\B}{\mathbb{B}}

\renewcommand{\Im}{{\operatorname{Im}}}

\begin{document}

\title{Backward orbits in the unit ball}

\author{Leandro Arosio and Lorenzo Guerini}
\address{Dipartimento Di Matematica, Universit\`a di Roma ``Tor Vergata'', Via Della Ricerca Scientifica 1, 00133 Roma, Italy}
\email{arosio@mat.uniroma2.it}
\address{Korteweg de Vries Institute for Mathematics, University of Amsterdam, Science Park 107, 1090GE Amsterdam, the Netherlands}
\email{lorenzo.guerini92@gmail.com}
\thanks{L.~Arosio was supported by SIR grant NEWHOLITE -- ``New methods in holomorphic iteration'', no.~RBSI14CFME}
\begin{abstract}
We show that, if $f\colon \mathbb{B}^q\to \mathbb{B}^q$ is a holomorphic self-map of the unit ball in $\mathbb{C}^q$ and $\zeta\in \partial \mathbb{B}^q$ is a boundary repelling fixed point with dilation $\lambda>1$, then there exists a backward orbit converging to $\zeta$ with step $\log \lambda$. Morever, any two backward orbits converging to the same boundary repelling fixed point  stay at finite distance. As a consequence there exists a unique canonical pre-model $(\mathbb{B}^k,\ell, \tau)$ associated with $\zeta$ where $1\leq k\leq q$, $\tau$ is a hyperbolic automorphism of $\mathbb{B}^k$, and  whose image $\ell(\B^k)$ is precisely the set of starting points of backward orbits with bounded step converging to $\zeta$. This answers questions in \cite{Os} and \cite{Ar1,Ar2}.
\end{abstract}

\maketitle

\section{Introduction}
For a holomorphic self-map $f\colon \mathbb{B}^q\to \mathbb{B}^q$ of the unit ball of $\mathbb{C}^q$,  there is a strong interplay among the notions of  boundary repelling fixed points, backward orbits with bounded step and pre-models.

We start by recalling some definitions and elementary properties. Denote by  $k_{\B^q}$ the Kobayashi distance of the ball. A fundamental property of $k_{\B^q}$ is the following generalization to several complex variables of the classical Schwarz-Pick Lemma:  for every holomorphic self map $f\colon \B^q\rightarrow \B^q$ we have
$$
k_{\B^q}(f(z),f(w))\le k_{\B^q}(z,w),\qquad\forall z,w\in\B^q.
$$
In particular every $\gamma\in Aut(\B^q)$ is an isometry with respect to $k_{\B^q}$.
 Recall that $Aut(\B^q)$ acts transitively on $\B^q$.

The {\sl Koranyi region} with vertex $\zeta\in\partial \B^q$ and amplitude $M>1$ is defined as
$$
K(\zeta,M):=\left\{z\in \mathbb{B}^q\colon k_{\mathbb{B}^q}(0,z)+\lim_{w\to \zeta}(k_{\mathbb{B}^q}(z,w)-k_{\mathbb{B}^q}(0,w))<2\log M\right\}.
$$

Koranyi regions are a several variables generalization of the classical Stolz regions in the unit disc, but if $q>1$ they are non-tangential to $\partial \B^q$ only in the complex normal direction, while they are tangent  to $\partial \B^q$ in the complex tangential directions. If  $\eta$ denotes the ray connecting the origin to $\zeta$, then Koranyi regions are comparable to the sets of the form $A(\eta,L):=\{z\in \B^q\colon k_{\B}(z,\zeta)<L\}$.  In particular a sequence converging to $\zeta$ is contained in a Koranyi region $K(\zeta,M)$ for some $M>1$ if and only if it is  contained in a region $A(\eta,L)$ for some $L>0$.

A point $\zeta\in \partial \mathbb{B}^q$ is a {\sl  boundary regular fixed point} if 
\begin{enumerate}
\item $K\hbox{-}\lim_{z\to \zeta} f(z)=\zeta,$
which means by definition that if $(z_n)$ is a sequence converging to $\zeta$ inside a Koranyi region $K(\zeta,M)$ then $f(z_n)\to \zeta$,
\item  the {\sl dilation} $\lambda$ of $f$ at $\zeta$, defined  as $$\log \lambda:=\liminf_{z\to \zeta}\left(k_{\mathbb{B}^q}(0,z)-k_{\mathbb{B}^q}(0,f(z))\right),$$ satisfies
$\lambda<+\infty$.
 If $\lambda>1$ (resp. $=1$, resp. $<1$) the point $\zeta$ is {\sl repelling}, (resp. {\sl indifferent}, resp. {\sl attracting}).
\end{enumerate}
A sequence $(z_n)$ in $\mathbb{B}^q$ is a {\sl backward orbit} if $f(z_{n+1})=z_n$ for all $n\geq 0$. The {\sl step} of $(z_n)$ is 
$\sigma(z_n):=\lim_{n\to +\infty}k_{\B^q}(z_n,z_{n+1})\in (0,+\infty].$
A {\sl pre-model} for $f$  is
a triple $(\Lambda,h,\varphi)$, where $\Lambda$ is a complex manifold called the {\sl base space}, $h\colon \Lambda\to \mathbb{B}^q$ is a 
holomorphic mapping called the {\sl intertwining mapping} and $\varphi\colon \Lambda \to \Lambda$ is an automorphism, such that the following diagram commutes:
$$\xymatrix{\Lambda\ar[r]^{\varphi}\ar[d]_{h}& \Lambda\ar[d]^{h}\\
\mathbb{B}^q\ar[r]^{f}& \mathbb{B}^q.}$$
We say that a pre-model  $(\Lambda,h,\varphi)$ is {\sl associated} with the boundary repelling point $\zeta$ if  for some (and hence for any) $x\in \Lambda$ we have $\lim_{n\to \infty}h(\varphi^{-n}(x))= \zeta$.

Poggi-Corradini  \cite{PoCo1} (see also Bracci \cite{Br})  showed, in the case of the   unit disc $\mathbb{D}\subset \mathbb{C}$, that given  a boundary repelling fixed point $\zeta\in \partial \mathbb{D}$ 
one can find a backward orbit with  step $\log \lambda$ converging to $\zeta$  and use such orbit to obtain an essentially unique pre-model $(\mathbb{D},h,\tau)$ associated with $\zeta$, where $\tau$ is a hyperbolic automorphism of the disc with dilation $\lambda$ at its repelling point.

This result was partially generalized by Ostapuyk \cite{Os} in the unit ball $\mathbb{B}^q$. She proved that given an {\sl isolated} 
 boundary repelling fixed point $\zeta$ one obtains with a similar method a pre-model
 $(\mathbb{D},h,\tau)$ associated with $\zeta$, where $\tau$ is a hyperbolic automorphism of the disc with dilation $\lambda$ at its repelling point.
Since such pre-model has no uniqueness property and is one-dimensional, it is asked in \cite[Question 8]{Os} what is the structure of the {\sl stable subset} $\mathcal{S}(\zeta)$, that is the subset of starting points of backward orbits with bounded step converging to $\zeta$, and whether one can find a ``best possible'' pre-model associated with $\zeta$.

In recent works \cite{Ar1,Ar2} a partial answer to such questions was given using the theory of canonical pre-models (see also \cite{ArBr}). 
 To state such results we need to introduce some definitions. 
 An automorphism $\tau$ of the ball $\B^q$ without inner fixed points (that is, non-elliptic) has either one or  two fixed points at the boundary. If $\tau$ has one fixed point it is called {\sl parabolic}, and the fixed point is indifferent. If $\tau$ has two fixed points, then it is called {\sl hyperbolic}, and in this case   one of the fixed points is repelling with dilation $\mu>1$ and the other  one is attracting with dilation $\frac{1}{\mu}$. A hyperbolic automorphism $\tau $ is conjugated to the following automorphism of the Siegel model for the unit ball $\mathbb{H}^q=\left\{(z,w)\in \C\times \C^{q-1}, \Im(z)>\|w\|^2\right\},$ 
$$\tau(z,w)= \left(\frac{1}{\mu} z,\frac{e^{it_1}}{\sqrt \mu}w_1,\dots, \frac{e^{it_{q-1}}}{\sqrt \mu}w_{q-1} \right),$$
where $t_j\in \mathbb{R}$ for $1\leq j\leq q-1$.

If $(y_n)$ is a backward orbit for $f$, we denote $[y_n]$ the family of all backward orbits $(z_n)$ for $f$ such that the sequence $(k_{\mathbb{B}^q}(z_n, y_n))$ is bounded.
If $(\hat\Lambda,\hat h,\hat\varphi)$ and $(\Lambda,h,\varphi)$ are pre-models for $f$, a {\sl morphism} $\hat\eta\colon (\hat\Lambda,\hat h,\hat\varphi)\to(\Lambda,h,\varphi)$ is given by a holomorphic map $\eta\colon\hat\Lambda\rightarrow \Lambda$ such that the following diagram commutes: 
\SelectTips{xy}{12}
\[ \xymatrix{\hat\Lambda\ar[rrr]^{\hat h}\ar[rd]^\eta\ar[dd]^{\hat\varphi} &&& \mathbb{B}^q \ar[dd]^f\\
& \Lambda \ar[rru]^{h} \ar[dd]^(.25)\varphi
\\
\hat\Lambda\ar'[r][rrr]^(.25){\hat h} \ar[rd]^\eta &&& \B^q\\
& \Lambda \ar[rru]^{h}.}
\]
In other words,  there exists a morphism $\hat\eta\colon (\hat\Lambda,\hat h,\hat\varphi)\to(\Lambda,h,\varphi)$ if and only if the pre-model
$ (\hat\Lambda,\hat h,\hat\varphi)$ ``factors through'' the pre-model $(\Lambda,h,\varphi)$. If $\eta\colon \hat\Lambda\to \Lambda$ is a biholomorphism, we say that $\hat \eta$ is an {\sl isomorphism}.

It was shown in  \cite{Ar1,Ar2} that every class $[y_n]$  of backward orbits with bounded step converging to $\zeta$ gives rise in a natural way to a 
{\sl canonical} pre-model $(\B^k,\ell,\tau)$, where $ k$ is an integer  satisfying $1\leq k\leq q$ and possibly depending on the class $[y_n]$,
and  $\tau$ is a hyperbolic automorphism of the ball $\mathbb{B}^k$  with dilation $\mu\geq \lambda$ at its repelling point. Such model  satisfies $\ell(\tau^{-n}(x))\in [y_n]$  for all $x\in \B^k$, and  is thus associated with $\zeta$.
Moreover the canonical pre-model  satisfies the following universal property: if $(\Lambda,h,\varphi)$ is a pre-model  such that
$h(\varphi^{-n}(x))\in [y_n]$ for all $x\in \Lambda$, then there exists a unique morphism $\hat\eta\colon (\Lambda,h,\varphi)\to(\B^k,\ell,\tau)$.  It is easy to see that any pre-model
 $(\hat\Lambda,\hat h,\hat\varphi)$ such that $\hat h(\hat \varphi^{-n}(x))\in [y_n]$ for all $x\in \hat\Lambda$
  and which satisfies the same universal property has to be isomorphic to $(\B^k,\ell,\tau)$, hence the name ``canonical''.

The following questions were left open in  \cite{Ar1,Ar2}.
\begin{enumerate}
\item Does a non-isolated  boundary repelling fixed point admit a backward orbit with bounded step converging to it (and hence an associated canonical pre-model)?
\item Is it  possible that  a boundary repelling fixed point $\zeta$ is associated with two distinct canonical pre-models, corresponding to two different classes of backward orbits with bounded step converging to $\zeta$? 
\item Is the dilation of a canonical pre-model associated with $\zeta$ always equal to $\lambda$?
\end{enumerate}
In the main result of this paper we give an answer to these three questions, showing that  every boundary repelling fixed point $\zeta$ is associated  with exactly one  (up to isomorphisms) canonical pre-model $(\mathbb{B}^k, \ell,\tau)$, which has dilation $\lambda$.  The dimension  $k$ of the canonical pre-model is thus a dynamical invariant naturally associated with the fixed point $\zeta$.
\begin{theorem}\label{main}
Let $f\colon \mathbb{B}^q\to \mathbb{B}^q$ be a holomorphic self-map, and let $\zeta\in \partial \mathbb{B}^q$ be a  boundary repelling fixed point with dilation $\lambda>1$. 
Then there exist an integer $1\le k\le q$ and a pre-model $(\mathbb{B}^k, \ell,\tau)$ associated with $\zeta$ such that
\begin{enumerate}
\item $\tau$ is a hyperbolic automorphism of $\B^q$ with dilation $\lambda$ at its repelling point $R$,
\item $\ell(\B^k)$ coincides with the stable subset $\mathcal S(\zeta)$,
\item Universal property: if  $(\Lambda,h,\varphi)$ is a pre-model associated with $\zeta$, then there exists a morphism $\hat\eta\colon (\Lambda,h,\varphi)\to (\mathbb{B}^k, \ell,\tau)$,
\item $K\hbox{-}\lim_{z\to R}\ell(z)=\zeta.$

\end{enumerate}
%, where $1\leq k\leq q$ and  $\tau$ is a hyperbolic automorphism of $\mathbb{B}^k$ with dilation $\lambda$ at its repelling point $R$. The pre-model $(\mathbb{B}^k, \ell,\tau)$ satisfies the following universal property: if  $(\Lambda,h,\varphi)$ is a pre-model associated with $\zeta$, then there exists a morphism
%$\hat\eta\colon (\Lambda,h,\varphi)\to (\mathbb{B}^k, \ell,\tau)$.
%Moreover, $K\hbox{-}\lim_{z\to R}\ell=\zeta.$

\end{theorem}
Theorem \ref{main} follows from \cite[Theorem 1.3]{Ar2} once we prove the following two results.
\begin{theorem}\label{main1}
Let $f\colon \mathbb{B}^q\to \mathbb{B}^q$ be a holomorphic self-map, and let $\zeta\in \partial \mathbb{B}^q$ be a  boundary repelling fixed point with dilation $\lambda>1$. 
Then there exists a backward orbit $(z_n)$ with step $\log\lambda$ converging to $\zeta$.
\end{theorem}
\begin{proposition}\label{main2}
Let  $(x_n)$ and $(y_n)$ be two  backward orbits with bounded step, both converging to the boundary repelling fixed point $\zeta\in \partial \mathbb{B}^q$. Then 
$$
\lim_{n\to\infty}k_{\B^q}(x_n,y_n)< \infty.
$$
\end{proposition}
Indeed, these two results imply that the family of backward orbits with bounded step converging to $\zeta$ is not empty and is a unique equivalence class $[y_n]$. The pre-model $(\mathbb{B}^k, \ell,\tau)$ is then the canonical pre-model  given by $[y_n]$.

The method of proof of Theorem \ref{main1} is inspired by the proofs in \cite{PoCo1,Os}.
To get rid of the problems posed by boundary repelling points close to $\zeta$ we use horospheres to define stopping times of the iterative process instead of euclidean balls centered at $\zeta$. This approach thus works also when the boundary fixed point $\zeta$ is not isolated. On the other hand additional work has to be done to show that the iterative process still converges to a backward orbit.

\section{Proof of Theorem \ref{main1}}
Without loss of generality we may assume that $\zeta=e_1=(1,0,\dots,0)$.
Recall that the {\sl horosphere} of center $e_1$, pole 0,  and radius $R>0$ is defined as
\begin{align*}E_0\left(e_1,R\right)&:=\left\{z\in \mathbb B^q\colon \,\frac{|1-(z,e_1)|^2}{1-\Vert z\Vert^2}<R\right\}\\
&=\left\{z\in \mathbb B^q\colon \lim_{w\to e_1}\left(k_{\mathbb{B}^q}(z,w)-k_{\mathbb{B}^q}(0,w)\right)<\log R\right\}.
\end{align*}

For all $k\in \mathbb{Z}$, let $r_k=\left(\frac{\lambda^k-1}{\lambda^k+1},0,\dots,0\right)$ and denote   $E_k:=E_0\left(e_1,\frac{1}{\lambda^k}\right)$. Notice that $r_k\in\partial E_k$. We have that (see e.g. \cite{Os})
\begin{equation}\label{boundedstep}
\lim_{k\to\infty}k_{\mathbb{B}^q}(r_k,f(r_k))=\log\lambda.
\end{equation}
Furthermore, by the Julia's lemma (see e.g. \cite{Ab}) we have
\begin{equation}
\label{Julia}
f(E_k)\subset E_{k-1}.
\end{equation}

If $f$ has no interior fixed points, then it admits a Denjoy-Wolff point $p$ such that the sequence $(f^n)$ converges to $p$ uniformly on compact subsets. Since the dilation of $f$ at $p$ is  less than or equal to $1$, it immediately follows that $p$ is different from $e_1$. If $f$ admits interior fixed points, it admits a limit manifold $M$ which is a holomorphic retract of $\mathbb B^q$ (see e.g. \cite[Theorem 2.1.29]{Ab}). By \cite[Proposition 3.4]{AbBr} it follows  that $e_1\not\in \overline M$. Hence,  by conjugating the map $f$ with an automorphism of the ball fixing $e_1$, we may further assume that 
 $$ M \cap \overline E_0=\varnothing.$$ Hence, in both cases, for every $k\geq 0$ there exists a first $n(k)\geq 0$ so that $f^{n(k)}(r_k)\not \in  \overline E_0$. By \eqref{Julia} we have that $n(k)> k$. We write $z_k:=f^{n(k)}(r_k)$.

\begin{lemma}\label{boundedaway}
The sequence $(z_k)$ is bounded away from $e_1$.
\end{lemma}

\begin{proof}[Proof of Lemma \ref{boundedaway}]
Suppose instead that there exists a subsequence $z_{k_i}\to e_1$. Let $b_i=z_{k_i}$ and $a_i=f^{n(k_i)-1}(r_{k_i})\in \overline E_0$. Since $k_{\mathbb{B}^q}(a_i,b_i)\le k_{\mathbb{B}^q}(r_{k_i},f(r_{k_i}))$ it follows by (\ref{boundedstep}) that
\[
\limsup_{i\to \infty}k_{\mathbb{B}^q}(a_i,b_i)\le \log\lambda.
\] 
It follows that the sequence $(a_i)$ also converges to $e_1$.

By the definition of $\lambda$, we also have
\begin{align*}
\liminf_{i\to\infty}k_{\mathbb{B}^q}(a_i,b_i)&\ge\liminf_{i\to\infty}\left(k_{\mathbb{B}^q}(0,a_i)-k_{\mathbb{B}^q}(0,b_i)\right)\\
&\geq \liminf_{z\to e_1}\left(k_{\mathbb{B}^q}(0,z)-k_{\mathbb{B}^q}(0,f(z))\right)\\
&= \log\lambda.
\end{align*}
We conclude that 
\begin{equation}\label{distancelimit}
\lim_{i\to\infty}k_{\mathbb{B}^q}(a_i,b_i)=\log\lambda.
\end{equation} 

%Let $\gamma_i\in Aut(\mathbb B^q)$ such that $\gamma_i(a_i)=0$ and $\gamma_i(e_1)=e_1$. The automorphism $\gamma_i$ send horospheres with center $e_1$ into horospheres with center in $e_1$. Furthermore since $0\in\gamma_i(\overline E_0)$ it follows that $\overline E_0\subset \gamma_i(\overline E_0)$. Since $a_i\in \overline E_0\setminus \overline E_1$, for every $z\in\overline E_0$ we have
%\begin{align*}
%\log h_{e_1,0}(\gamma_i(z))&=\lim_{w\to e_1}[k_{\B^q}(\gamma_i(z),w)-k_{\B^q}(0,w)]\\
%&=\lim_{w\to e_1}[k_{\B^q}(\gamma_i(z),w)-k_{\B^q}(\gamma_i(0),w)]+\lim_{w\to e_1}[k_{\B^q}(\gamma_i(0),w)-k_{\B^q}(0,w)]\\
%&=\log h_{e_1,0}(z)-\log h_{e_1,0}(a_1),
%\end{align*}
%which implies that $ \log h_{e_1,0}(\gamma_i(z))\le \log\lambda$, proving that
%\[
%\overline E_0\subset \gamma_i(\overline E_0)\subset \overline E_{-1}.
%\]

Let  $\gamma_i\in Aut(\mathbb B^q)$ be such that $\gamma_i(a_i)=0$ and $\gamma_i(e_1)=e_1$. Such an automorphism can be obtained
composing a parabolic automorphism fixing $e_1$ and a hyperbolic automorphism fixing $e_1$ and $-e_1$, with dilation $\mu$ at $e_1$ which satisfies $1\leq\mu<\lambda$ since $a_i\in \overline E_0\setminus \overline E_1$. Hence the dilation of $\gamma_i$ at $e_1$ is $\mu$, which implies that 
$$\overline E_0\subset \gamma_i(\overline E_0)\subset \overline E_{-1}.$$

Since $k_{\mathbb{B}^q}(0,a_i)\to\infty$, it follows that $k_{\mathbb{B}^q}(\gamma_i(0),0)\to \infty$. Since the sequence $(\gamma_i(0))$ is  contained in $\overline E_{-1}$,   we obtain $\gamma_i(0)\to e_1$.

 Let $0<\alpha<1$ and set
$$
c_i:=-\alpha\frac{\gamma_i(0)}{\Vert\gamma_i(0)\Vert}.
$$
If we write $\beta:=\log\frac{1+\alpha}{1-\alpha}$, then for every $i$ we have $k_\B(0,c_i)=\beta$. Since $\gamma_i(0)\to e_1$ it follows that $c_i\to c_\infty=(-\alpha, 0,\dots ,0)$.

%Let $\alpha>0$. Define the sequence $(c_i)$ in $\mathbb{B}^q$ as follows: $c_i$ is the point on the real geodesic connecting $\gamma_i(0)$ with $0=\gamma_i(a_i)$ which lies after $0$ and such that $k_{\mathbb{B}^q}(0,c_i)=\alpha$. Since $\gamma_i(0)\to e_1$ it follows that $c_i\to c_\infty=(\beta, 0,\dots ,0)$ with $\beta$ real and strictly negative.

\begin{figure}[ht]
\centering
\begin{tikzpicture}[line cap=round,line join=round,>=triangle 45,x=1.0cm,y=1.0cm]
\clip(-3.5,-3.5) rectangle (3.5,3.5);
\draw(0,0) circle (3cm);
\fill[gray!40] (1.5,0) circle (1.5);
\draw (1.5,0) circle (1.5);
\draw (-1.95,-0.06)-- (0,0);
\draw (0,0)-- (2.9,0.09);
\draw (0,0)-- (0.8,1.77);
\draw [dash pattern=on 3pt off 3pt] (0,0)-- (3,0);
\draw [shift={(-3.45,5.17)}] plot[domain=4.99:5.61,variable=\t]({1*5.45*cos(\t r)+0*5.45*sin(\t r)},{0*5.45*cos(\t r)+1*5.45*sin(\t r)});
\draw [shift={(2.93,2.28)}] plot[domain=3.37:4.7,variable=\t]({1*2.18*cos(\t r)+0*2.18*sin(\t r)},{0*2.18*cos(\t r)+1*2.18*sin(\t r)});
\begin{scriptsize}
\fill  (0,0) circle (1pt);
\draw (-0.1,-0.2) node {$0$};
\fill  (3,0) circle (1pt);
\draw (3.3,0) node {$e_1$};
\fill (2.9,0.09) circle (1pt);
\draw (2.7,0.3) node {$\gamma_i(0)$};
\fill (-1.95,-0.06) circle (1pt);
\draw (-2.15,-0.1) node {$c_i$};
\fill (0.8,1.77) circle (1pt);
\draw (0.8,2) node {$\gamma_i(b_i)$};
\draw (1.5,-1.3) node {$E_0$};
\end{scriptsize}
\end{tikzpicture}
\caption{The position of the points at the $i$-th step, when $q=1$.}
\end{figure}
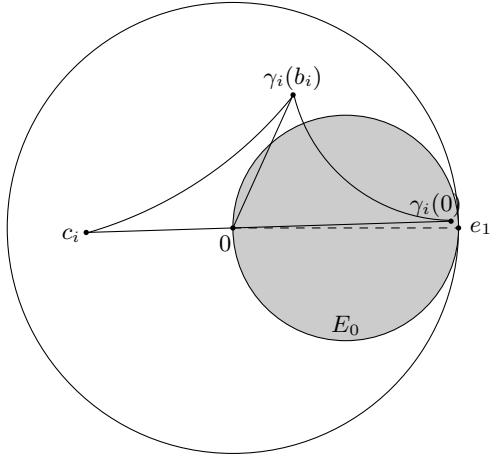

By (\ref{distancelimit}) the sequence $\gamma_i(b_i)$ is relatively compact in $\mathbb B^q$. By taking a subsequence of $k_i$ if necessary, we may assume that $\gamma_i(b_i)\to b_\infty\in \mathbb{B}^q$. Notice that since $b_i\not\in \overline E_0$, we must also have $\gamma_i(b_i)\not\in \overline E_0$. It follows that $b_\infty\not\in E_0$.

We claim that $k_{\mathbb{B}^q}(c_\infty,b_\infty)<\log\lambda+\beta$. Indeed, since $k_{\mathbb{B}^q}(0,b_\infty)=\log\lambda$ and $k_\B(0,c_\infty)=\beta$, we get by triangular inequality that $k_{\mathbb{B}^q}(c_\infty,b_\infty)\leq\log\lambda+\beta$. Equality holds if and only if $b_\infty$ is contained in the real geodesic connecting the origin to $e_1$. 
%{
%\color{red} Dati tre punti $a,b,c\in \B^q$ tali che $\kappa(a,b)=\kappa(a,c)+\kappa(b,c)$, possiamo supporre che $a=0$ e $b=(b_1,0,\dots,0)$, per qualche $b_1\in (0,1)$. Sia $\pi:\B^q\rightarrow\Delta$ la proiezione sulla prima coordinata e $\rho$ la metrica di Poincar\'e sul disco. La distanza di Kobayashi viene contratta dalla mappa $\pi$, ovvero
%$$
%\rho(\pi(x),\pi(y))\le \kappa(x,y), 
%$$
%in particolare, se $c=(c_1,\gamma)$, otteniamo che $\rho(b_1,c_1)\le\kappa(b,c)$. Dal momento che
%$$
%\kappa(0,x)=\rho(0,\Vert x\Vert),
%$$
%segue anche che $\kappa(0,b)=\rho(0,b_1)$ e $\kappa(0,c)\ge\rho(0,c_1)$. Se per assurdo $\gamma\neq 0$, allora avremmo $\kappa(0,c)>\rho(0,c_1)$. In questo caso possiamo concludere che
%$$
%\rho(0,b_1)=\kappa(0,b)=\kappa(0,c)+\kappa(b,c)>\rho(0,c_1)+\rho(b_1,c_1),
%$$
% conaleil contraddice le disuguaglianza triangolare. Possiamo concludere che $\gamma=0$. Dal momento che la restrizione della mappa $\pi$ al disco $\Delta\cdot e_1$ \'e un isometria, usando argomenti in una variabile, possiamo concludere che i tre punti devono essere sulla stessa geodetica.
%}
But this is not possible since such geodesic is contained in the horosphere $ E_0$.
Let $\delta>0$ be such that  $k_{\mathbb{B}^q}(c_\infty,b_\infty)< \log\lambda +\beta-2\delta.$

By the last inequality and by the definition of $\lambda$ we may choose $i$ big enough such that $k_{\mathbb{B}^q}(c_i,\gamma_i(b_i))<\log\lambda+\beta -\delta$ and  $$ k_{\mathbb{B}^q}(\gamma_i(0),0)-k_{\mathbb{B}^q}(\gamma_i(0),\gamma_i(b_i))=k_{\mathbb{B}^q}(0,a_i)-k_{\mathbb{B}^q}(0,b_i)\ge \log\lambda-\delta.$$ We conclude that
\begin{align*}
k_{\mathbb{B}^q}(\gamma_i(0),c_i)-k_{\mathbb{B}^q}(\gamma_i(0),\gamma_i(b_i))&=k_{\mathbb{B}^q}(\gamma_i(0),0)+k_{\mathbb{B}^q}(0,c_i)-k_{\mathbb{B}^q}(\gamma_i(0),\gamma_i(b_i))\\
&\geq \beta +\log\lambda-\delta\\
&>k_{\mathbb{B}^q}(c_i,\gamma_i(b_i)),
\end{align*}
contradicting the triangular inequality. 
\end{proof}

We are now ready to conclude the proof of Theorem \ref{main1}.
 Since the sequence $(z_k)=(f^{n(k)}(r_k))$ is bounded away from $e_1$ and contained in $\overline E_{-1}$, we can extract a subsequence $k_0(h)$ such that  
$(f^{n(k_0(h))}(r_{k_0(h)}))$ converges  to a point $w_0\in \mathbb{B}^q$. Now consider the sequence $(f^{n(k_0(h))-1}(r_{k_0(h)}))$.
Since
$$k_{\mathbb{B}^q}(f^{n(k_0(h))}(r_{k_0(h)}),f^{n(k_0(h))-1}(r_{k_0(h)}))\leq k_{\mathbb{B}^q}(r_{k_0(h)},f(r_{k_0(h)}))\to\log\lambda,$$
we can extract a subsequence $k_1(h)$ of $k_0(h)$ such that $(f^{n(k_1(h))-1}(r_{k_1(h)}))$ converges to a point $w_1\in \mathbb{B}^q\cap \overline E_0$ and 
$f(w_1)=w_0$. Iterating this procedure we obtain for all $j\geq 1$
 a subsequence $(k_{j}(h))$ of $(k_{j-1}(h))$ such that $(f^{n(k_{j}(h))-j}(r_{k_{j}(h)}))$ converges to a point $w_{j}\in  \mathbb{B}^q\cap \overline E_0$ and $f(w_{j})=w_{j-1}$. We notice that, since $n(k)>k$, the expression $(f^{n(k_{j}(h))-j}(r_{k_{j}(h)}))$ is well defined for $h$ large enough.
 
The backward orbit $(w_j)$ has bounded step since, for all $j\geq 0$,
\begin{align*}
k_{\mathbb{B}^q}(w_{j-1},w_{j})&=\lim_{h\to\infty}k_{\mathbb{B}^q}(f^{n(k_{j}(h))-{j-1}}(r_{k_{j}(h)}),f^{n(k_{j}(h))-j}(r_{k_{j}(h)})) \\&\leq \lim_{h\to\infty} k_{\mathbb{B}^q}(r_{k_{j}(h)},f(r_{k_{j}(h)}))\\
&= \log\lambda.
\end{align*}
We are left with showing that $w_j\to e_1$. Since the sequence $(w_j)_{j\geq 1}$ is contained in  $\overline E_0$ it is enough to show that  there is no   subsequence $(w_{m(j)})$ converging to a point  $x\in \mathbb{B}^q$. Assume by contradiction that such a subsequence exists. Then there exists a compact set $K\subset  \overline E_0\cap \mathbb{B}^q$ containing the sequence $(w_{m(j)})$. Recall that if $f$ has no interior fixed points its Denjoy-Wolff point is different from $e_1$, and that if $f$ has fixed points then the limit manifold $M$ does not intersect $\overline E_0$. Hence there exists an integer $N\geq 0$ such that 
$f^n(K)\cap K=\varnothing$ for all $n\geq N$. But this is a contradiction since $(w_j)$ is a backward orbit.

%
%\begin{proposition}
%There exists a backward orbit $(x_n)$ which converges to $e_1$ and satisfies
%\[
%\lim_{n\to\infty}k(x_n,x_{n+m})=\frac{m}{2}\log\lambda,\qquad\forall m\ge 0.
%\]
%In particular the backward orbit $(x_n)$ has bounded step. 
%\end{proposition}
%\begin{proof}
%By taking a subsequence of $r_k$ if necessary we may assume that for every $n$ the sequence $f^{n_k-n}(r_k)$, defined only for $k$ sufficiently big, converges to $x_n\in \overline{\mathbb B}^m$. By the previous lemma we have $x_0\in \mathbb B^m$. By \eqref{boundedstep} we can find $L>0$ such that $k(r_k,f(r_k))<L$ for every $k$, furthermore when $k$ is sufficiently big, it follows that
%\[
%k(f^{n_k-n}(r_k),f^{n_k}(r_k))\le k(r_k,f^n(r_k))\le nL.
%\]
%We conclude that, once $n$ is fixed, the distance $k(x_0,f^{n_k-n}(r_k))$ is bounded, which implies that also $x_n\in\mathbb B^m$. By the continuity of $f$, we conclude that $f(x_{n+1})=x_n$.
%
%Given $m\ge 1$, the sequence $k(x_n,x_{n+m})$ is not decreasing, furthermore by \eqref{boundedstep} we have
%\begin{align*}
%\lim_{n\to\infty}k(x_n,x_{n+m})&=\lim_{n\to\infty}\lim_{k\to\infty}k(f^{n_k-n}(r_k),f^{n_k-n-m}(r_k))\\
%&\le\lim_{n\to\infty}\lim_{k\to\infty}k(f^m(r_k),r_k)\\
%&\le\frac{m}{2}\log\lambda.
%\end{align*}
%On the other hand we also have
%\begin{align*}
%\frac{m}{2}\log\lambda&=\liminf_{z\to e_1}k(0,z)-k(0,f^m(z))\\
%&\le \liminf_{z\to e_1} k(f^m(z),z)\\
%&\le \lim_{n\to\infty}k(x_n,x_{n+m}),
%\end{align*}
%which proves the desired limit.
%\end{proof}

\section{Proof of Proposition \ref{main2}}

Given a backward  orbit $(x_n)$ one can always assume that it is indexed by integers $n\in\mathbb Z$, defining for all $n\geq 0$, $x_{-n}:= f^{n}(x_0)$.
\begin{lemma}\label{trasl}
Let $(x_n)$ and $(y_n)$ be two  backward orbits with bounded step, both converging to $e_1$. Then $\lim_{n\to+\infty}k_{\mathbb{B}^q}(x_n,y_n)< \infty$ if and only if
\[
\lim_{n\to+\infty}\inf_{m\in\mathbb Z}k_{\mathbb{B}^q}(x_n,y_m)<\infty.
\]
\end{lemma}
\begin{proof}

Notice first that  $\inf_{m\in\mathbb Z}k_{\mathbb{B}^q}(x_n,y_m))$ is non-decreasing in $n$, therefore the limit for $n\to+\infty$ exists (possibly infinite). Suppose now that there exists $C>0$ such that 
\[
\inf_{m\in\mathbb Z}k_{\mathbb{B}^q}(x_n,y_m)<C,\qquad\forall n\in\mathbb Z.
\]
The forward orbit of the point $y_0$ is bounded away from the bondary repelling fixed point $e_1$. Since $\lim_{m\to+\infty}x_m=e_1$, we may assume that there exists $N>0$ such that $k_{\mathbb{B}^q}(x_N,y_m)\ge C$ for all $m<0$. Given such $N$,  since $\lim_{m\to+\infty}y_m=e_1$, we may also find $M>0$ such that $k_{\mathbb{B}^q}(x_N,y_m)\ge C$ for all $m>M$.
For every $n\ge N$ we may find an integer $\alpha(n)$ such that $k_{\mathbb{B}^q}(x_n,y_{\alpha(n)})<C$. By the properties of the Kobayashi distance we deduce that
\[
k_{\mathbb{B}^q}(x_N,y_{\alpha(n)-n+N})<C,
\]
which implies that $-N\le\alpha(n)-n\le M-N$. It follows that we may find a divergent sequence $n_k\ge N$ so that $\alpha(n_k)-n_k=\alpha\in\mathbb Z$. Notice that for every $k$ we have
\[
k_{\mathbb{B}^q}(x_{n_k},y_{n_k+\alpha})=k_{\mathbb{B}^q}(x_{n_k},y_{\alpha(n_k)})<C,
\]
which implies that for every $n\in\mathbb Z$, we have $k_{\mathbb{B}^q}(x_n,y_{n+\alpha})<C$. Finally since $(y_n)$ has bounded step we deduce that
\[
k_{\mathbb{B}^q}(x_n,y_n)\le k_{\mathbb{B}^q}(x_n,y_{n+\alpha})+k_{\mathbb{B}^q}(y_{n+\alpha},y_n)\le C+|\alpha|\sigma(y_n).
\]
The other implication is trivial.
\end{proof}

The following lemma is essentially contained in \cite{Os}, even if not explicitly stated.
\begin{lemma}\label{Ostapyuk}
Let $(z_n)$ be a backward orbit with bounded step converging to $e_1$. Then there exists $M>0$ so that
$$
z_n\in K_0(\tau,M),\qquad\forall n\ge 0.
$$
\end{lemma}
\begin{proof}
By the definition of dilation $\lambda$ we have
$$
\liminf_{n\to\infty}\left( k_{\B^q}(0,z_{n+1})- k_{\B^q}(0,z_{n})\right)\ge \log \lambda,
$$
and therefore that
$$
\liminf_{n\to\infty} \frac{1-\Vert z_n\Vert}{1-\Vert z_{n+1}\Vert}\ge\lambda.
$$
If we write $t_n:=1-\Vert z_n\Vert$, and take $\lambda^{-1}<c<1$ we conclude that there exists $n_0\in\mathbb N$ such that 
$$
t_{n+1}\le ct_n,\qquad \forall n\ge n_0.
$$
By shifting the sequence $(z_n)$ if necessary, we may assume that $n_0=0$. We conclude that 
\begin{equation}
\label{formuletta}
t_{n+k}\le c^kt_n,\qquad\forall n,k\ge 0.
\end{equation}
The result of the lemma follows from \eqref{formuletta} exactly as in the elliptic case of \cite[Theorem 1.8]{Os}.
\end{proof}

Let  $(x_n)_{n\in \mathbb{Z}}$ and $(y_n)_{n\in \mathbb{Z}}$ be two  backward orbits with bounded step, both converging to the boundary repelling fixed point $e_1\in \partial \mathbb{B}^q$. Recall that we denote by $\eta$ the ray connecting the origin to $e_1$. By Lemma \ref{Ostapyuk} the sequence  $(x_n)_{n\in\mathbb{N}}$ is contained in a Kor\'anyi region, and thus it is contained in a  region $A(\eta,L):=\{z\in \B^q\colon k_{\B^q}(z,\eta)<L\}$ for some $L>0$ (see e.g. \cite[Lemma 2.5]{BrGePoCo}). We claim that there exists $R>0$ such that 
$$A(\eta,L)\subset \left\{ z\in \mathbb{B}^q\colon \inf_{m\in\mathbb{Z}} k_{\mathbb{B}^q}(z,y_m)<R\right\}.$$
Once the claim is proved, the result follows by Lemma \ref{trasl}.

It is enough to show that there exists a constant $S>0$ such that for all $w\in \eta,$ $\inf_{m\in \mathbb{Z}} k_{\mathbb{B}^q}(w,y_m)<S$.
Since  $(y_n)_{n\in \mathbb{N}}$  is contained in a Kor\'anyi region,  there exists a constant $C>0$ such that 
$k_{\mathbb{B}^q}(y_n, \eta)< C$ for all $n\in \mathbb{N}$.
Let $a_n$ be a point in $\eta$ such that $k_{\mathbb{B}^q}(a_n,y_n)<C$. Clearly $a_n\to e_1$. Let $w$ be a point in the portion of $\eta$ which connects $a_0$ to $e_1$. There exists $n(w)$ such that $w$ belongs to the portion of $\eta$ which connects $a_n$ to $a_{n+1}$.
Hence $$\inf_m k_{\mathbb{B}^q}(w, y_m)\leq C+k_{\mathbb{B}^q}(a_n,a_{n+1})\leq C+2C+ k_{\mathbb{B}^q}(y_n,y_{n+1})\leq 3C+ \sigma(y_n),$$
concluding the proof of Proposition \ref{main2}.


\begin{thebibliography}{99}
 \bibitem{Ab} M. Abate, {\sl Iteration theory of holomorphic maps on taut manifolds}, Mediterranean Press, Rende (1989).
\bibitem{AbBr} M. Abate, F. Bracci, {\sl Common boundary regular fixed points for holomorphic semigroups in strongly convex domains},  Contemporary Mathematics {\bf 667} (2016), 1--14. 
\bibitem{Ar1} L. Arosio, {\sl The stable subset of a univalent self-map}, Math. Z. {\bf 281} (2015), no. 3-4, 1089--1110.
\bibitem{Ar2} L. Arosio, {\sl Canonical models for the forward and backward iteration of holomorphic maps}, J. Geom. Anal. {\bf 27} (2017), no.2, 1178--1210.
\bibitem{ArBr} L. Arosio, F. Bracci, {\sl  Canonical models for holomorphic iteration}, Trans. Amer. Math. Soc. {\bf 368} (2016), no.5,  3305--3339.
\bibitem{Br} F. Bracci, {\sl Fixed points of commuting holomorphic mappings other than the Wolff point}, Trans. Amer. Math. Soc. {\bf 355} (2003), no. 6, 2569--2584.
\bibitem{BrGePoCo}  F. Bracci, G. Gentili, P. Poggi-Corradini, {\sl Valiron's construction in higher dimensions}, Rev. Mat. Iberoam. \textbf{26} (2010), no. 1, 57--76.
\bibitem{Os}
O.  Ostapyuk, {\sl Backward iteration in the unit ball}, Illinois J. Math. {\bf 55} (2011), no. 4, 1569--1602.
\bibitem{PoCo1}
P. Poggi-Corradini,
{\sl Canonical conjugations at fixed points other than the Denjoy--Wolff point}
Ann. Acad. Sci. Fenn. Math. {\bf 25} (2000), no. 2, 487--499.
\bibitem{PoCo2}
P. Poggi-Corradini, {\sl Backward-iteration sequences with bounded hyperbolic steps for analytic self-maps of the disk} Rev. Mat. Iberoamericana {\bf 19} (2003), no. 3, 943--970.

\end{thebibliography}
\end{document}